\documentclass{amsart}

\usepackage{amssymb,amsmath,latexsym,amsthm}
\usepackage{MnSymbol}
\usepackage{graphicx}
\usepackage{fontenc}%
\usepackage{mathrsfs}
\usepackage{color}

\usepackage[verbose]{wrapfig}
\usepackage[english]{babel} 
\usepackage{pstricks-add} 

\newtheorem{theorem}{Theorem}[section]
\newtheorem{definition}[theorem]{Definition}
\newtheorem{lemma}[theorem]{Lemma}
\newtheorem{example}[theorem]{Example}

\usepackage{pstricks,pst-text,pst-grad,pst-node,pst-3dplot,pstricks-add,pst-poly}
\definecolor{pink}{rgb}{1, .75, .8}
\psset{arrows=->, labelsep=3pt, mnode=circle}

\definecolor{lgrey}{gray}{.85}

\def\defineTColor#1#2{%
 \newpsstyle{#1}{%
  fillstyle=vlines,hatchcolor=#2,
  hatchwidth=0.1\pslinewidth,
  hatchsep=1\pslinewidth}%
  }
\defineTColor{Tgray}{gray}
\defineTColor{Tgreen}{green}
\defineTColor{Tyellow}{yellow}
\defineTColor{Tred}{red} 
\defineTColor{Tblue}{blue} 
\defineTColor{Tmagenta}{magenta}
\defineTColor{Torange}{orange}
\defineTColor{Twhite}{white}
\defineTColor{Tgray}{lgrey}
\defineTColor{Tbgreen}{brightgreen}

\begin{document}

\title[Strongly far proximity and hyperspace topology]{Strongly far proximity and hyperspace topology}

\author[J.F. Peters]{J.F. Peters$^{\alpha}$}
\email{James.Peters3@umanitoba.ca, cguadagni@unisa.it}
\address{\llap{$^{\alpha}$\,}Computational Intelligence Laboratory,
University of Manitoba, WPG, MB, R3T 5V6, Canada and
Department of Mathematics, Faculty of Arts and Sciences, Ad\i yaman University, 02040 Ad\i yaman, Turkey}
\author[C. Guadagni]{C. Guadagni$^{\beta}$}
\address{\llap{$^{\beta}$\,}Computational Intelligence Laboratory,
University of Manitoba, WPG, MB, R3T 5V6, Canada and
Department of Mathematics, University of Salerno, via Giovanni Paolo II 132, 84084 Fisciano, Salerno , Italy}
\thanks{The research has been supported by the Natural Sciences \&
Engineering Research Council of Canada (NSERC) discovery grant 185986.}

\subjclass[2010]{Primary 54E05 (Proximity); Secondary 54B20 (Hyperspaces)}

\date{}

\dedicatory{Dedicated to the Memory of Som Naimpally}

\begin{abstract}
This article introduces strongly far proximity $\mathop{\not{\delta}}\limits_{\mbox{\tiny$\doublevee$}}$, 
which is associated with Lodato proximity $\delta$.  A main result in this paper is the introduction of a hit-and-miss topology on $\mbox{CL}(X)$, the hyperspace of nonempty closed subsets of $X$, based on the strongly far proximity.
\end{abstract}

\keywords{Hit-and-Miss Topology, Hyperspaces, Proximity, Strongly Far}

\maketitle

\section{Introduction}
Usually, when we talk about proximities, we mean \textit{Efremovi\v{c} proximities}.  Nearness expressions are very useful and also represent a powerful tool because of the relation existing among \textit{Efremovi\v c proximities}, \textit{Weil uniformities} and $\mbox{T}_2$ compactifications. But sometimes \textit{Efremovi\v c proximities} are too strong. So we want to distinguish between a weaker and a stronger forms of proximity.  For this reason, we consider at first \textit{Lodato proximity} $\delta$ and then, by this, we define a stronger proximity by using the Efremovi\v c property related to proximity.

\section{Preliminaries}
Recall how a \textit{Lodato proximity} is defined~\cite{Lodato1962,Lodato1964,Lodato1966} (see, also, \cite{Naimpally2009,Naimpally1970}).

\begin{definition} 
Let $X$ be a nonempty set. A \textit{Lodato proximity $\delta$} is a relation on $\mathscr{P}(X)$ which satisfies the following properties for all subsets $A, B, C $ of $X$ :
\begin{itemize}
\item[P0)] $A\ \delta\ B \Rightarrow B\ \delta\ A$
\item[P1)] $A\ \delta\ B \Rightarrow A \neq \emptyset $ and $B \neq \emptyset $
\item[P2)] $A \cap B \neq \emptyset \Rightarrow  A\ \delta\ B$
\item[P3)] $A\ \delta\ (B \cup C) \Leftrightarrow A\ \delta\ B $ or $A\ \delta\ C$
\item[P4)] $A\ \delta\ B$ and $\{b\}\ \delta\ C$ for each $b \in B \ \Rightarrow A\ \delta\ C$
\end{itemize}
Further $\delta$ is \textit{separated }, if 
\begin{itemize}
\item[P5)] $\{x\}\ \delta\ \{y\} \Rightarrow x = y$.
\end{itemize}
\end{definition}

\noindent When we write $A\ \delta\ B$, we read $A$ \emph{is near to} $B$ and when we write $A \not \delta B$ we read $A$ \emph{is far from} $B$.
A \emph{basic proximity} is one that satisfies $P0)-P3)$.
\textit{Lodato proximity} or \textit{LO-proximity} is one of the simplest proximities. We can associate a topology with the space $(X, \delta)$ by considering as closed sets the ones that coincide with their own closure, where for a subset $A$ we have
\[
\mbox{cl} A = \{ x \in X: x\ \delta\ A\}.
\]
This is possible because of the correspondence of Lodato axioms with the well-known Kuratowski closure axioms. 

By considering the gap between two sets in a metric space ( $d(A,B) = \inf \{d(a,b): a \in A, b \in B\}$ or $\infty$ if $A$ or $B$ is empty ), Efremovi\v c introduced a stronger proximity called \textit{Efremovi\v c proximity} or \textit{EF-proximity}.  

\begin{definition}
An \emph{EF-proximity} is a relation on $\mathscr{P}(X)$ which satisfies $P0)$ through $P3)$ and in addition 
\[A \not\delta B \Rightarrow \exists E \subset X \hbox{ such that } A \not\delta E \hbox{ and } X\setminus E \not\delta B \hbox{ EF-property.}\]
\end{definition}

A topological space has a compatible EF-proximity if and only if it is a Tychonoff space.

Any proximity $\delta$ on $X$ induces a binary relation over the powerset exp $X,$ usually denoted as $ \  \ll_\delta $ and  named    the  {\it   natural strong inclusion associated with } $\delta,$ by declaring that $ A$ is {\it strongly included} in $B, \ A \ll_{\delta} B, \ $ when $A$ is far from the complement of $B,  \ A \not\delta X \setminus B .$

By strong inclusion the \textit{Efremivi\v c property} for $ \delta $ can be written also as a betweenness property   \

 \centerline {  (EF) \  \   \  \  If $A \ll_{\delta} B,$  then there exists some $C$ such that $A \ll_{\delta} \  C \ll_{\delta} \ B$ .} \  \ \

A pivotal example of \emph{EF-proximity} is the \emph{metric proximity} in a metric space $(X,d)$ defined by 
\[
A\ \delta\ B \Leftrightarrow d(A,B) = 0.
\]
That is, $A$ and $B$  {\it either intersect or are asymptotic}: for each natural number $n$ there  is a  point  $a_n$ in $A $ and a point  $b_n$ in $B$ such that $d(a_{n},b_{n}) < \frac{1}{n}$.\\

\subsection{Hit and far-miss topologies}
Let $\mbox{CL}(X)$ be  the hyperspace of all non-empty closed subsets of  a space  $X.$
 {\it Hit and miss} and {\it hit and far-miss}  topologies on $\mbox{CL}(X)$  are obtained by the join of two halves. Well-known examples are Vietoris topology~\cite{Vietoris1921,Vietoris1922,Vietoris1923,Vietoris1927} (see, also,~\cite{Beer1993,Beer1993hyperspace,DiConcilio2013action,DiConcilio2000SetOpen,DiConcilio2000PartialMaps,DiConcilio1989,DiMaio1992hypertop,Guadagni2015,Naimpally2002,Som2006hypertopology,Som2013aps}) and Fell topology~\cite{Fell1962HausdorfTop,Guadagni2015,Beer1993}.  In this article, we concentrate on an extension of Vietoris based on the strongly far proximity. \\
 
\noindent \underline{\emph{Vietoris topology}}\\
\smallskip
Let $X$ be an Hausdorff space. The \textit{Vietoris topology} on $\mbox{CL}(X)$ has as subbase all sets of the form
\begin{itemize}
\item $V^- = \{E \in \mbox{CL}(X): E \cap V \neq \emptyset\}$, where $V$ is an open subset of $X$,
\item $W^+ = \{C \in \mbox{CL}(X): C \subset W \}$, where $W$ is an open subset of $X$.
\end{itemize}

The topology ${\tau_V}^-$ generated by the sets of the first form is called \textbf{hit part} because, in some sense, the closed sets in this family hit the open sets $V$. Instead, the topology 
${\tau_V}^+$ generated by the sets of the second form is called \textbf{miss part}, because the closed sets here miss the closed sets of the form $X \setminus W$.\\
The Vietoris topology is the join of the two part: $\tau_V = {\tau_V}^- \vee {\tau_V}^+$. It represents the prototype of hit and miss topologies.\\
The Vietoris topology was modified by Fell. He left the hit part unchanged and in the miss part, ${\tau_F}^+$ instead of taking all open sets $W$, he took only open subsets with compact complement.\\
\underline{\emph{Fell topology:}}  \ \  \  \  $\qquad \  \qquad  \ \ \   \tau_F = {\tau_V}^- \vee {\tau_F}^+$\\

It is possible to consider several generalizations. For example, instead of taking open subsets with compact complement, for the miss part we can look at subsets running in a family of closed sets $\mathscr{B}$. So we define the {\it hit and   miss topology on $\mbox{CL}(X)$ associated with} ${\mathscr{B}}$ as the topology generated by the join of the hit sets $ A^{-},$  where  $A$ runs over all  open subsets of $X$, with the miss sets $A^{+}$, where   $A$ is once again an open subset of $X,$ but more,  whose  complement runs in   $\mathscr{B}$.

Another kind of generalization concerns the substitution of the inclusion present in the miss part with a strong inclusion associated to a proximity. Namely, when the space $X$ carries a proximity $\delta,$ then a proximity variation of the miss part can be displayed by replacing the miss sets with {\it far-miss sets} \ $A^{++} :=  \{ \ E \in \mbox{CL}(X) : E \ll_{\delta}   A  \ \}.$

Also in this case we can consider $A$ with the complement running in a family $\mathscr{B}$ of closed subsets of $X$. Then the {\it hit and far-miss topology , $\tau_{\delta, \mathscr{B}}$, associated with $\mathscr{B}$ } is generated by the join of the hit sets  $A^{-},$   where $A$ is open,   with far-miss sets $A^{++},$ where the complement of $A$ is in $\mathscr{B}$.

Fell topology can be considered as well an example of hit and far-miss topology. In fact, in any EF-proximity, when a compact set is contained in an open set, it is also strongly contained.   

\setlength{\intextsep}{0pt}
\begin{wrapfigure}[8]{R}{0.45\textwidth}
\begin{minipage}{4.5 cm}
\begin{center}
\begin{pspicture}
 (0.0,1.5)(2.5,3.5)
\psframe[linecolor=black](-0.8,0.5)(4.5,3.0)
\pscircle[linecolor=black,linestyle=dotted,linewidth=0.05,dotsep=0.05,style=Tgray](0.58,1.55){0.88}
\pscircle[linecolor=black,linestyle=solid,linewidth=0.05,style=Tgreen](0.38,1.35){0.38}
\pscircle[linecolor=black,linestyle=dotted,linewidth=0.05,dotsep=0.05,style=Tgray](3.38,1.85){1.00}
\pscircle[linecolor=black,linestyle=solid,linewidth=0.05,style=Tgreen](3.68,1.65){0.48}
\rput(-0.5,2.8){\footnotesize  $\boldsymbol{X}$}
\rput(0.38,1.35){\footnotesize $\boldsymbol{A}$}
\rput(0.58,2.25){\footnotesize $\boldsymbol{C}$}
\rput(3.68,1.65){\footnotesize $\boldsymbol{B}$}
\rput(3.08,2.55){\footnotesize  $\boldsymbol{E}$}
\rput(0.28,0.00){\qquad\qquad\qquad\qquad\footnotesize 
                 $\boldsymbol{\mbox{Fig.}\ 3.1.\  A\ \stackrel{\not{\text{\normalsize$\delta$}}_e}{\text{\tiny$\doublevee$}}\ B}$}
 \end{pspicture}
\end{center}
\end{minipage}
\end{wrapfigure}
\setlength{\intextsep}{2pt}

\section{Main Results}
Results for the \emph{strongly far} proximity~\cite{Peters2015visibility} (see, also,~\cite{Peters2015VoronoiAMSJ,Peters2014KleePhelps,Peters2012notices}) are given in this section.  Let $X$ be a nonempty set and $\delta$ be a \emph{Lodato proximity} on $\mathscr{P}(X)$.

\begin{definition}
We say that $A$ and $B$ are $\delta-$strongly far and we write $\mathop{\not{\delta}}\limits_{\mbox{\tiny$\doublevee$}}$ if and only if $A \not\delta B$ and there exists a subset $C$ of $X$ such that $A \not\delta X \setminus C$ and $C \not\delta B$, that is the Efremovi\v c property holds on $A$ and $B$.
\end{definition}

\begin{example}
In the Fig. 3.1, let $X$ be a nonempty set endowed with the Euclidean metric proximity $\delta_e$ ($e$ for Euclidean), $C,E\subset X, A\subset C, B\subset E$.  Clearly, $A\ \stackrel{\not{\text{\normalsize$\delta$}}_e}{\text{\tiny$\doublevee$}}\ B$ ($A$ is strongly far from $B$), since $A \not\delta_e B$ so that  $A \not\delta_e X \setminus C$ and $C \not\delta_e B$.  Also observe that the Efremovi\v{c} property holds on $A$ and $B$. \qquad \textcolor{black}{$\blacksquare$}
\end{example}

Observe that $A \not \delta B$ does not imply $A \mathop{\not{\delta}}\limits_{\mbox{\tiny$\doublevee$}} B$. In fact, this is the case when the proximity $\delta$ is not an \emph{EF-proximity}. 

\begin{example}\label{Alex}
Let $(X, \tau)$ be a non-locally compact Tychonoff space. The \emph{Alexandroff proximity} is defined as follows: \ $A\ \delta_A\ B \Leftrightarrow \mbox{cl}A \cap \mbox{cl}B \neq \emptyset$ or both $\mbox{cl}A$ and $\mbox{cl}B$ are non-compact.  This proximity is a compatible Lodato proximity that is not an EF-proximity. So $A \not \delta_A B$ does not imply $A
\stackrel{\not{\text{\normalsize$\delta$}}_A}{\text{\tiny$\doublevee$}} 
B$. \qquad \textcolor{black}{$\blacksquare$}
\end{example}

\begin{theorem}
The relation $\mathop{\not{\delta}}\limits_{\mbox{\tiny$\doublevee$}}$
is a basic proximity.
\end{theorem}
\begin{proof}
Immediate from the properties of $\delta$.
\end{proof}

We can also view the concept of strong farness in many other ways. For example, let $A\stackrel{\not{\hat{\text{\normalsize$\delta$}}}}{\text{\tiny$\doublevee$}} B$, read $A$ \emph{$\hat{\delta}$-strongly far from} $B$, defined by
\[ A \stackrel{\not{\hat{\text{\normalsize$\delta$}}}}{\text{\tiny$\doublevee$}} B \Leftrightarrow \exists E, C \subset X : \ A \subset \mbox{int}(\mbox{cl}E), \ B \subset
\mbox{int}(\mbox{cl}C) \hbox{ and } \mbox{int}(\mbox{cl}E) \cap \mbox{int}(\mbox{cl}C)= \emptyset .\]
This form of strong farness is illustrated in Fig. 3.1.  The relation $\stackrel{\not{\hat{\text{\normalsize$\delta$}}}}{\text{\tiny$\doublevee$}}$ appears to be stronger than $\stackrel{\not{\text{\normalsize$\delta$}}}{\text{\tiny$\doublevee$}}$, but it is possible to observe the following relations.

\begin{theorem}
The relation $\stackrel{\not{\text{\normalsize$\delta$}}}{\text{\tiny$\doublevee$}}$ is stronger than $\stackrel{\not{\hat{\text{\normalsize$\delta$}}}}{\text{\tiny$\doublevee$}}$, that is $A \stackrel{\not{\text{\normalsize$\delta$}}}{\text{\tiny$\doublevee$}} B \Rightarrow A \stackrel{\not{\hat{\text{\normalsize$\delta$}}}}{\text{\tiny$\doublevee$}} B.$ 
\end{theorem}
\begin{proof}
Suppose $A \stackrel{\not{\text{\normalsize$\delta$}}}{\text{\tiny$\doublevee$}} B$. This means that there exists a subset $C$ of $X$ such that $A \not\delta X \setminus C$ and $C \not\delta B$. By the Lodato property $P4)$ (see \cite{Lodato1962}), we obtain that $\mbox{cl}A \cap \mbox{cl}(X \setminus C)= \emptyset $ and $\mbox{cl}C \cap \mbox{cl}B =\emptyset$. So $\mbox{cl}A \subset \mbox{int}(C), \ \mbox{cl}B \subset \mbox{int}(\mbox{cl}(X\setminus C ))$ and $\mbox{int}(C) \cap \mbox{int}(\mbox{cl}(X\setminus C))= \emptyset$, that gives $A \stackrel{\not{\hat{\text{\normalsize$\delta$}}}}{\text{\tiny$\doublevee$}} B$.
\end{proof}

We now want to consider \emph{hit and far-miss topologies} related to $\delta$ and $\stackrel{\text{\normalsize$\delta$}}{\text{\tiny$\doublevee$}}$ on $\mbox{CL}(X)$, the hyperspace of nonempty closed subsets of $X$.

To this purpose, call $\tau_\delta$ the topology having as subbase the sets of the form:
\begin{itemize}
\item $V^- = \{E \in \mbox{CL}(X): E \cap V \neq \emptyset\}$, where $V$ is an open subset of $X$,
\item $A^{++} =  \{ \ E \in \mbox{CL}(X) : E \not\delta X\setminus  A  \ \}$, where $A$ is an open subset of $X$.
\end{itemize}

and $\tau_\doublevee$ the topology having as subbase the sets of the form:
\begin{itemize}
\item $V^- = \{E \in {\rm \mbox{CL}}(X): E \cap V \neq \emptyset\}$, where $V$ is an open subset of $X$,
\item $A_\doublevee =  \{ \ E \in \mbox{CL}(X) : E \stackrel{\not{\text{\normalsize$\delta$}}}{\text{\tiny$\doublevee$}} X\setminus  A  \ \}$, where $A$ is an open subset of $X$
\end{itemize}

It is straightforward to prove that these are admissible topologies on $\mbox{CL}(X)$.\\
The following results concern comparisons between them.   From this point forward, let $X$ be a T$_1$ topological space.

\begin{lemma}\label{lemma1}
Let $A, B, C \in \mbox{\rm CL}(X) $. If $A \not\delta B \Rightarrow A \mathop{\not{\delta}}\limits_{\mbox{\tiny$\doublevee$}} B$ for all $A \in \mbox{\rm CL}(X)$, then $C \subseteq B$. That is $(X \setminus B)^{++} \subseteq (X \setminus C)_\doublevee \Rightarrow C \subseteq B$.
\end{lemma}
\begin{proof}
By contradiction, suppose $C \not\subseteq B$. Then there exists $x \in C: x \not\in B$. So $x \not\delta B$ but $x \mathop{\delta}\limits_{\mbox{\tiny$\doublevee$}} C$, which is absurd.
\end{proof}

\begin{lemma}\label{lemma2}
Let $\delta= \delta_A$, the Alexandroff proximity on a non-locally compact Tychonoff space, and let $H$ and $E$ be open subsets of $X$. Then $H_W \subseteq E^{++} \Leftrightarrow H \subseteq E$.
\end{lemma}
\begin{proof}
$"\Rightarrow"$. By contradiction, suppose that $H \not\subseteq E$. Then we can choose $X\setminus H$ as compact subset and $X \setminus E$ non-compact. Take another closed subset $B$ non compact and suppose $B \stackrel{\not{\text{\normalsize$\delta$}}_A}{\text{\tiny$\doublevee$}} X \setminus H$. So there exists $D: B \not{\delta}_A X \setminus D$ and $D \not{\delta}_A X \setminus H$, and this is compatible with the previous choices. But $B\ \delta_A\ X \setminus E$, being both non-compact sets.\\
$"\Leftarrow"$. For any $B \in \mbox{CL}(X)$, $B \stackrel{\not{\text{\normalsize$\delta$}}_A}{\text{\tiny$\doublevee$}} X \setminus H \Rightarrow B \stackrel{\not{\text{\normalsize$\delta$}}_A}{\text{\tiny$\doublevee$}} X \setminus E \Rightarrow B \not\delta X \setminus E. $
\end{proof}

Now let $\tau_\delta^{++}$ be the hypertopology having as subbase the sets of the form $A^{++}$, where $A$ is an open subset of $X$, and let $\tau_\doublevee^+$ the hypertopology having as subbase the sets of the form $A_\doublevee$, again with $A$ an open subset of $X$.

\begin{theorem}
The hypertopologies $\tau_\delta^{++}$ and $\tau_\doublevee^+$ are not comparable.
\end{theorem}
\begin{proof}
First we want to prove that, in general, $\tau_\doublevee^+ \not\subset \tau_\delta^{++}$. Consider the space of rational numbers $X = \mathbb{Q} $ and the \emph{ Alexandroff proximity} $\delta_A$ (see example \ref{Alex}). Let $H$ be an open subset of $X$ with $cl(X\setminus H)$ non-compact and suppose $E \in H_\doublevee $, with $E \in \mbox{CL}(X)$. We ask if there exists a $\tau_\delta^{++}-$open set, $K^{++}$, such that $E \in K^{++} \subseteq H_\doublevee $. We have two cases: $\mbox{cl}(X \setminus K)$ compact or not. First, suppose $\mbox{cl}(X \setminus K)$ compact and $A \in K^{++}$ with $\mbox{cl}A$ non-compact. Then it must be $\mbox{cl}A \cap \mbox{cl}(X\setminus K)= \emptyset$. But $A \stackrel{\text{\normalsize$\delta$}_A}{\text{\tiny$\doublevee$}} X \setminus H$, because for all $
D, \ A\ \delta_A\ X\setminus D$ or $D\ \delta_A\ X \setminus H$. In fact if $\mbox{cl}D$ is compact, then $\mbox{cl}(X\setminus D)$ is not compact. So either both $\mbox{cl}A$ and $\mbox{cl}(X\setminus D)$ are non-compact, or both $\mbox{cl}D$ and $\mbox{cl}(X\setminus H)$ are non-compact. Instead, suppose $\mbox{cl}(X \setminus K)$ non-compact. So, being $A \not\delta_A X \setminus K$, we have $\mbox{cl}A$ compact and $\mbox{cl}A \cap \mbox{cl}(X \setminus K)=\emptyset$. To obtain $A \stackrel{\not{\text{\normalsize$\delta$}}_A}{\text{\tiny$\doublevee$}} X\setminus H$, by lemma \ref{lemma1} we should have $K \subseteq H$. So we need a set $K$ such that $\mbox{cl}A \subseteq K \subseteq H$ and more with $\mbox{cl}K$ compact and $\mbox{cl}A \subseteq K \subseteq \mbox{cl}K \subseteq H$. But we are in a non-locally compact space, so it could be not possible.

Conversely, we want to prove that $\tau_\delta^{++} \not\subset \tau_\doublevee^+  $. Consider again the space of rational numbers $X = \mathbb{Q} $ and the \emph{ Alexandroff proximity} $\delta_A$. Take $E^{++} \in \tau_\delta^{++} $ and $A \in E^{++}$, with $E$ open subset of $X$. To identify a $\tau_\doublevee^+$-open set, $H_\doublevee$, such that $A \in H_\doublevee \subset E^{++}$, by lemma \ref{lemma2}, we need $H \subseteq E$. But we can choose $A$ and $X\setminus E$ in such a way that EF-property does not hold. So EF-property does not hold either for $A$ and $X \setminus H$, for each $H \subset E$. Hence $A$ cannot belong to any $H_\doublevee$ included in $ E^{++}$.
\end{proof}


\end{document}